\documentclass[10pt]{amsart}
\pdfoutput=1 

\usepackage{amsmath,amsthm,amssymb}
\usepackage[mathscr]{eucal}
 \usepackage{cite}
\usepackage{upgreek}
\usepackage[bookmarks=false]{hyperref}
\usepackage{enumerate}
\usepackage{mathrsfs}
\usepackage{blindtext}
\usepackage{scrextend}
\usepackage{bm} 
\addtokomafont{labelinglabel}{\sffamily}
\usepackage{color}
\usepackage[at]{easylist}
\usepackage{tikz}
\usepackage{svg}
\usepackage{graphicx}
\usepackage{mathtools}

\usepackage{comment}

\numberwithin{equation}{section}

\newtheorem{main}{Theorem}

\newtheorem{mcor}[main]{Corollary}

\newtheorem{thm}{Theorem}[section]
\newtheorem*{thm*}{Theorem}
\newtheorem{lem}[thm]{Lemma}
\newtheorem*{prob*}{Problem}

\newtheorem*{prop*}{Proposition}

\newtheorem*{cor*}{Corollary}

\theoremstyle{definition}
\newtheorem{defn}[thm]{Definition}
\newtheorem*{defn*}{Definition}

\newtheorem*{question*}{Question}
\newtheorem*{Pquestion*}{Popa's question}

\newtheorem*{conv*}{Convention}

\newcommand{\N}{\mathbb{N}}
\newcommand{\R}{\mathbb{R}}
\newcommand{\C}{\mathbb{C}}

\newcommand{\F}{\mathbb{F}}

\newcommand{\M}{\mathbb{M}}

\newcommand{\cF}{\mathcal{F}}

\newcommand{\cU}{\mathcal{U}}

\newcommand{\ee}{\varepsilon}

\definecolor{aws}{RGB}{255, 0, 221}

\begin{document}

\title[Lifting for N-independent sets in II$_1$ factors]
{Lifting for N-independent sets in II$_1$ factors}

%\begin{comment}

\author[G. Patchell]{Gregory Patchell}
\address{\parbox{\linewidth}{Department of Mathematics, Vanderbilt University,\\ 1326 Stevenson Center, Nashville, TN 37240, USA}}
\email{gregory.d.patchell@vanderbilt.edu}
\urladdr{https://sites.google.com/view/gpatchel}

\author[A. Shiner]{Austin Shiner}
\address{\parbox{\linewidth}{Mathematical Institute, University of Oxford, Andrew Wiles Building, \\ Radcliffe Observatory Quarter, Woodstock Road, Oxford, OX2 6GG, UK}}
\email{austin.shiner@mansfield.ox.ac.uk}

\begin{abstract}
    In a tracial von Neumann algebra $(M,\tau)$, two subsets $X,Y$ are said to be $N$-independent if alternating products of length at most $N$ of trace zero words from $X$ and $Y$ have trace zero. We show that two approximately $2N$-independent finite tuples can be perturbed, in operator norm, to be exactly $N$-independent. Consequently, we are able to generalize Theorem 5.1 and Theorem H of Houdayer--Ioana \cite{houdayer2023asymptotic} and resolve Problem 31 of Kunnawalkam Elayavalli \cite{elayavalli2025} in the positive.
\end{abstract}

\maketitle

\section{Introduction}

Free independence has long been a central theme in operator algebras. Already in 1943, Murray and von Neumann used freeness to show that $L(\F_2)$ is not approximately finite dimensional \cite{MvN43}. In the 1980s, Voiculescu formally defined free independence in pursuit of the study of the free group factors (see \cite{voiculescu2006symmetries}), giving rise to the field of free probability, which nowadays interacts with subfactors, $C^*$-algebras, random matrix theory, combinatorics, quantum groups, quantum information theory, and statistical inference. Very recently, free independence in the ultrapowers of $C^*$-algebras, dubbed selflessness by Robert in \cite{robert2025selfless}, has led to a flurry of work on the structure of $C^*$-algebras (e.g., \cite{amrutam2025strict,ozawa2025proximality}). One inspiration for selflessness was Popa's asymptotic freeness theorem from 1995: in any separable II$_1$ factor $M$, there is a unitary $u\in M^\cU$ such that $M$ and $uMu^*$ are freely independent \cite{popa1995free}. This result has led to numerous developments in the theory of II$_1$ factors, including the existence of non-elementarily equivalent non-Gamma factors \cite{chifan2023exotic} and the SOT-contractibility of the unitary group in any II$_1$ factor \cite{jekel2025unitary}.

In \cite{chifan2023exotic} (see also the related works \cite{elayavalli2025sequential,patchell20253,gao2026exotic}), one requires an orthogonality assumption to run the construction. (We note that orthogonality is equivalent to 1-independence in our terminology.) In \cite{chifan2023exotic}, orthogonality in the ultrapower of order 3 and order 2 unitaries is shown to lift to orthogonal unitaries of the same orders. In \cite{houdayer2023asymptotic}, they are able to remove the finite  order assumption on the unitaries, but at a cost: now, 2-independence is a required assumption. That is, 2-independent unitaries lift to orthogonal unitaries. Various other lifting problems are very natural. Free independence cannot lift to free independence, as a consequence of Popa's asymptotic freeness theorem (for a recent exposition, see \cite{boulanger2025free}). Whether orthogonality lifts to orthogonality remains open in general. But we are able to resolve the following: a $2N$-independent sequence of operators lifts to an $N$-independent sequence of operators. This in particular resolves Problem 31 of \cite{elayavalli2025}. 

\begin{main}\label{mthm-sri-problem}
    Let $M$ be a II$_1$ factor. Two freely independent Haar unitaries $u,v$ in $M^\cU$ can be lifted to sequences of unitaries $(u_n)$, $(v_n)$ such that for all $N$, $W^*(u_n)$ and $W^*(v_n)$ are $N$-independent for all $n$ sufficiently large.
\end{main}

Our results show more concretely that if two finite sets $X,Y$ of trace-zero elements are ``almost'' $2N$-independent, then there is a unitary operator $u$, close to 1 in operator norm, such that $uXu^{-1}$ and $Y$ are $N$-independent. Thanks to the aforementioned Popa's asymptotic freeness theorem, we obtain the following natural generalization of \cite[Theorem~H]{houdayer2023asymptotic}.

\begin{main}\label{mthm-perturb}
    Let $M$ be a II$_1$ factor. If $X,Y\subset M\ominus \C$ are finite sets and $N\ge1$, there is a unitary $u\in M$ such that $uXu^{-1}$ and $Y$ are $N$-independent.
\end{main}

In a similar manner, we are able to generalize \cite[Theorem~5.1]{houdayer2023asymptotic} to the setting of $N$-independent abelian algebras.

\begin{main}\label{mthm-lift-algs}
    Let $(M_i,\tau_i)_{i \in I}$ be a set of tracial von Neumann algebras, and let $\cU$ be an ultrafilter on $I$. Let $A,B$ be abelian, separable subalgebras of $\prod_\cU M_i$ which are $2N$-independent. Then there exist finite-dimensional $N$-independent abelian subalgebras $C_i,D_i\subset M_i$ such that $A\subset \prod_\cU C_i$ and $B\subset \prod_\cU D_i$. 
\end{main}

Our techniques only require holomorphic functional calculus, so our results apply not just to tracial von Neumann algebras but also to tracial C$^*$-algebras and even tracial Banach $*$-algebras. Our results also give an analogue of the Amitsur--Levitski theorem, which states that the $n$-by-$n$ matrix algebra over $\C$ satisfies no noncommutative polynomial identity of degree less than $2n$ (though it does satisfy an identity of degree $2n$) \cite{amitsur1950minimal}. Our results, combined with Voiculescu's asymptotic free independence theorem \cite{voiculescu1991limit}, abstractly show the following.

\begin{mcor}
    Let $m,n,N$ be positive integers. Then there is a positive integer $K$ such that there is an $m$-tuple $x \in \M_{K\times K}(\C)^m$ and an $n$-tuple $y \in \M_{K\times K}(\C)^n$ such that $x$ and $y$ are $N$-independent. 
\end{mcor}

\subsection*{Acknowledgements}

The authors thank Srivatsav Kunnawalkam Elayavalli for helpful discussions around $N$-independence.

\subsection*{Funding}

The authors were supported by the Engineering and Physical Sciences Research Council (UK), grant EP/X026647/1.  

\subsection*{Open Access, Data, and AI Statement}

For the purpose of Open Access, the authors have applied a CC BY public copyright license to any Author Accepted Manuscript (AAM) version arising from this submission. Data sharing is not applicable to this article as no new data were created or analyzed in this work. No LLMs were used in the preparation of this manuscript.

\section{Main Theorem}

\begin{defn} Let $(M,\tau)$ be a tracial von Neumann algebra. Two subsets $X,Y \subset M \ominus \C$ are said to be \emph{$N$-independent} if for all $1\le k\le N$, $x_1,\ldots,x_k\in X$, and $y_1,\ldots,y_k\in Y,$ we have $\tau(x_1y_1x_2y_2\cdots x_ky_k)=0$ (we also allow $y_k = 1$, and if $k\geq 2,$ $x_1=1$). We say that two subalgebras $P,Q\subset M$ are $N$-independent if the sets $P\ominus \C$ and $Q\ominus \C$ are $N$-independent.

If $X$ and $Y$ are $N$-independent for all $N\ge1$, we say that $X$ and $Y$ are freely independent, and likewise for subalgebras $P$ and $Q$.
    
\end{defn}

\subsection*{Setup and notation.}

Let $(M,\tau)$ be a tracial von Neumann algebra. Fix $N\in \N.$ Denote by $M_{\mathrm{sa}}$ the elements $x\in M$ such that $x=x^*;$ denote by $M_{\mathrm{sa},1  }$ the subset of $M_{\mathrm{sa}}$ consisting of elements $x$ such that $\|x\|\le 1.$ Let $x = (x_1,\ldots,x_m)\in M_{\mathrm{sa},1}^m$ and $y = (y_1,\ldots,y_n)\in M_{\mathrm{sa},1}^n$. For a unitary $u\in M$, we write $uxu^{-1} = (ux_1u^{-1},\ldots,ux_mu^{-1})$.

Define $\cF_N$ as the set of all formal words 
\begin{align*}
    \kappa &={i_1}{j_1}{i_2}{j_2}\cdots {i_r}{j_r} \\
    & - {j_1}{i_2}\cdots {j_r}{i_1} \\
    &+ {i_2}{j_2}\cdots {j_r}{i_1}{j_1}\\
    &\mp \ldots \\
    &- {j_r}{i_1} \cdots {j_{r-1}}{i_r}.
\end{align*}
where $1\le r\le N$, $1\le i_k\le m$ and $1\le j_k \le n$ for all $k$. We define $\kappa(x,y)$ by taking a word $\kappa \in \cF_N$ and substituting each $i_k$ for $x_{i_k}$ and each $j_{k}$ for $y_{j_k}$. (In other words, take $x_{i_1}y_{j_1}x_{i_2}y_{j_2}\cdots x_{i_r}y_{j_r}$, cyclically permute the factors, and alternately add/subtract these terms together.) We can think of the elements $\kappa(x,y)$ for $\kappa \in \cF_N$ as ``generalized commutators''. We now define three quantities:
\begin{align*}
    \delta(x,y) &:= \min\{\|\kappa(x,y)\|_2 : \kappa \in \cF_N\}; \\
    \gamma(x,y) &:= \max\{|\tau(\kappa_2(x,y)^*\kappa_1(x,y))| : \kappa_1 \neq \kappa_2 \in \cF_N\};\\
    \ee(x,y) &:= \max\{|\tau(x_{i_1}y_{j_1}\cdots x_{i_r}y_{j_r})| : 1\le r\le N, 1\le i_s\le m, 1\le j_s\le n\}.
\end{align*}
We also define $\gamma(x,y) = 0$ if $|\mathcal{F}_N| = 1$.

For our first lemma, we slightly abuse notation and conflate ordered tuples of elements with finite sets for the parameters of $\ee.$

\begin{lem}\label{lem:trace-trick}
    Let $X,Y\subset M_{\mathrm{sa}}\ominus \C$ be finite sets. Set $X' = X\cup \{xx' - \tau(xx') : x,x'\in X\}$ and $Y' = Y \cup \{yy' - \tau(yy'): y,y'\in Y\}$. 
    \begin{enumerate}
        \item If $\ee(X',Y')=0$, then $X$ and $Y$ are $N$-independent.
        \item If $\mathrm{span}(X,1)$ and $\mathrm{span}(Y,1)$ are subalgebras of $M$ and $\ee(X,Y) = 0,$ then $X$ and $Y$ are $N$-independent.
    \end{enumerate}
\end{lem}

\begin{proof}
    (1) We must show that four types of words have trace 0: words starting with a letter in $X$ and ending with a letter in $X$; starting in $X$ and ending in $Y$; starting in $Y$ and ending in $X$; starting in $Y$ and ending in $Y$. By symmetry via taking adjoints, if suffices to consider the first two types of words. Since $X\subset X'$, $Y\subset Y'$, and $\ee(X',Y')=0$, it follows immediately that words starting in $X$ and ending in $Y$ have trace 0.

    Now consider a word $x_1y_1\cdots x_{k-1}y_{k-1}x_k$ with $1\le k\le N$ and each $x_i\in X$, $y_i\in Y$. We show by induction on $k$ that the trace of $x_1y_1\cdots x_{k-1}y_{k-1}x_k$ is 0. If $k=1,$ this is immediate since $X$ consists of trace zero elements (and similarly for $Y$). Otherwise, we compute that
    \begin{align*}
        \tau(x_1y_1\cdots x_{k-1}y_{k-1}x_k) &= \tau(x_kx_1y_1\cdots x_{k-1}y_{k-1}) \\
        &= \tau(x_kx_1)\tau(y_1\cdots x_{k-1}y_{k-1}) \\
        &\quad + \tau([x_kx_1 - \tau(x_kx_1)]y_1\cdots x_{k-1}y_{k-1}) \\
        &= 0 + 0
    \end{align*}
    since $$\tau(y_1\cdots x_{k-1}y_{k-1}) = 0$$ by the induction hypothesis and $$\tau([x_kx_1 - \tau(x_kx_1)]y_1\cdots x_{k-1}y_{k-1}) = 0$$ since $\ee(X',Y')=0.$

    (2) If $\mathrm{span}(X,1)$ and $\mathrm{span}(Y,1)$ are subalgebras of $M,$ note that $X' \subseteq \operatorname{span} X$ and $Y' \subseteq \operatorname{span}Y$, then  $\ee (X',Y')=0$ and apply (1).
\end{proof}

Our goal now is essentially to take $x,y$ where $\gamma(x,y)$ and $\ee(x,y)$ are small and find a unitary $u$, close in operator norm to 1, satisfying $\ee(uxu^{-1},y) = 0$. We begin by stating an auxiliary lemma which appears as Lemma 5.3 in \cite{houdayer2023asymptotic}.

\begin{lem}\label{lem-lin-indep}
    Let $(M,\tau)$ be a tracial von Neumann algebra, $\xi_1,\ldots,\xi_p\in M_{\mathrm{sa},1}$, and $\alpha_1,\ldots,\alpha_p\in \R$, for some $p\ge 1.$ Let $\delta>0$ and $\gamma \ge 0, \delta^2 - (p-1)\gamma>0$. Assume that $\|\xi_i\|_2 \ge \delta$ and $|\langle \xi_i,\xi_j\rangle| \le \gamma$ for all $1\le i,j\le p$ with $i\ne j$. Then there exists $h\in M_{\mathrm{sa}}$ such that $\tau(h\xi_i) = \alpha_i$ for all $1\le i \le p$ and $\|h\| \le \frac{\sum_{j=1}^p|\alpha_j|}{\delta^2-(p-1)\gamma}$. Moreover, $h$ can be chosen from the $\R$-span of $\xi_1,\ldots,\xi_p$.
\end{lem}

We now state and prove the analogue of \cite[Lemma~5.4]{houdayer2023asymptotic} in the setting of $N$-independence.

\begin{lem}\label{lem-main-technical}
    Fix $N.$ There exist constants $C_1,C_2$, and $C_3$, depending only on $N,$ with the following property. Let $(M,\tau)$ be a tracial von Neumann algebra, $x = (x_1,\ldots,x_m)\in M_{\mathrm{sa},1}^m$, and $y = (y_1,\ldots,y_n)\in M_{\mathrm{sa},1}^n$, for some $m,n\ge1.$ Set $\delta = \delta(x,y)$, $\ee= \ee(x,y),$ and $\gamma = \gamma(x,y)$. Assume that $4N|\cF_N|\ee < \delta^2 - (2|\cF_N|-1)\gamma$ and set $\lambda = \frac{4N|\cF_N|\ee}{\delta^2 - (2|\cF_N|-1)\gamma}$. Then there exists $u$ a unitary in $M$ such that 
    \begin{enumerate}
        \item $\|u-1\| \le 2\lambda$;
        \item $\delta(uxu^{-1},y) \ge \delta - C_1\lambda$;
        \item $\ee(uxu^{-1},y)\le C_2\lambda^2$;
        \item $\gamma(uxu^{-1},y) \le \gamma + C_3\lambda$.
    \end{enumerate}
\end{lem} 

\begin{proof}
    For each $\kappa\in \cF^{N}$, we note that there is another element $-\kappa^* \in \cF_N$ given by reversing the order of the products and multiplying by $-1$. (The notation is suggestive of the fact that $-\kappa^*(x,y) = -(\kappa(x,y))^*$ as $x$ and $y$ consist of self-adjoint elements.) Of course, one of two alternatives occur: either $\kappa=-\kappa^*$ or not. (For a simple example of the former, consider $\kappa = ij-ji$.) 

    Order the (finitely many) elements $\kappa \in \cF_N$ and iteratively define the following quantities. 
    If $\xi_\kappa$ and $\alpha_\kappa$ are already defined, skip to the next element in $\cF_N$. 
    Otherwise, if $\kappa=-\kappa^*,$ define $\xi_\kappa = -\frac{i}{2N}\kappa(x,y)$ and $\alpha_\kappa = \frac{1}{2N} \tau(x_{i_1}y_{j_1}\cdots x_{i_r}y_{j_r})$. Note that $\alpha_k$ is real, indeed since $\kappa = -\kappa^\ast$, the adjoint of the base word is a cyclic permutation of said base word, hence by the trace identity we have $\overline{\tau(w)}=\tau(w^\ast) = \tau(w)$.
    If $\kappa\neq -\kappa^*$, define $\xi_\kappa = \frac{1}{4N}(\kappa(x,y)+\kappa(x,y)^*)$,
    $\xi_{-\kappa^*} = \frac{i}{4N}(\kappa(x,y)-\kappa(x,y)^*)$, 
    $\alpha_\kappa = -\frac{1}{2N}\mathrm{Im}(\tau(x_{i_1}y_{j_1}\cdots x_{i_r}y_{j_r}))$,
    and $\alpha_{-\kappa^*} = -\frac{1}{2N}\mathrm{Re}(\tau(x_{i_1}y_{j_1}\cdots x_{i_r}y_{j_r}))$. We note that these quantities are well-defined since the words in $\cF_N$ are defined only up to (even) cyclic permutation and $\tau$ is invariant under cyclic permutations.

    We have that for each $\kappa\in \cF_N$, $\xi_\kappa = \xi_\kappa^*$, $\|\xi_\kappa\| \le 1$, and $$\|\xi_\kappa\|_2^2 \geq \frac{\delta^2-\gamma}{8N^2}.$$ 
    Indeed if $\kappa = - \kappa^\ast$, then $\xi_k = - \frac i {2N}z$ and $\|\xi_\kappa\|_2^2 = \frac{\|z\|_2^2}{4N^2}$. Since $\|z\|_2 > \delta$, we have $\|\xi_\kappa\|_2^2 \ge \frac{\delta^2}{4N^2} \ge \frac{\delta^2 - \gamma}{8N^2}$. Next, if $\kappa \ne -\kappa^\ast$ and $\kappa$ is defined first, then $\xi_k = \frac{z+z^\ast}{4N}$ and then 
    $$\|\xi_k\|_2^2 = \frac 1 {16N^2}\|z+z^\ast\|_2^2 = \frac 1 {16N^2}(2\|z\|_2^2 + 2 \operatorname{Re} \langle z,z^\ast \rangle)$$
    Noting that $|\langle z,z^\ast \rangle| \le \gamma$ and $\|z\|_2^2 \ge \delta^2$, we also get the desired $\|\xi_\kappa\|_2^2 \ge \frac{\delta^2 - \gamma}{8N^2}$. Lastly, if $\kappa \ne -\kappa^\ast$ and $-\kappa^\ast$ is defined first, then $\xi_{-\kappa^\ast} = \frac{i(z-z^\ast)}{4N}$ and then $\|\xi_{-\kappa^\ast}\|_2^2 = \frac 1 {16N^2}(2\|z\|_2^2 - 2 \operatorname{Re} \langle z,z^\ast \rangle) \ge \frac{\delta^2 - \gamma}{8N^2}$. A similar case analysis, using 
    $$|\langle \kappa(x,y)+\kappa(x,y)^*,i(\kappa(x,y)-\kappa(x,y)^*)\rangle| \le 2|\langle \kappa(x,y),\kappa(x,y)^*\rangle|$$
    also shows that $|\langle \xi_\kappa, \xi_{\kappa'}\rangle| \le \frac{\gamma}{4N^2}$ for all $\kappa \ne \kappa'$ in $\cF_N$.

    Applying Lemma \ref{lem-lin-indep}, we obtain $h\in M_{\mathrm{sa}}$ such that for each $\kappa\in \cF_N,$ we have $\tau(h\xi_\kappa) =\alpha_\kappa$ and $$\|h\| \le \frac{\sum_{\kappa\in\cF_N} |\alpha_\kappa|}{\frac{\delta^2-\gamma}{8N^2} - (|\cF_N|-1)\frac{\gamma}{4N^2}} \le \frac{4N|\cF_N|\ee}{(\delta^2-\gamma) - 2(|\cF_N|-1)\gamma} = \lambda < 1.$$

    Define $u = \exp(ih) = \sum_{n\ge0}\frac{1}{n!}(ih)^n$, which is a unitary in $M$. A routine calculation (using the triangle inequality and the submultiplicativity of $\|\cdot\|$) confirms that $\|u-1\| \le 2\lambda$ and $\|u - (1+ih)\| \le \lambda^2$.

    We also note the following inequality, which we will use repeatedly. Let $w(x,y)$ be a word in $x$ and $y,$ that is, $w(x,y) = x_{i_1}y_{j_1}\cdots x_{i_r}y_{j_r}$ (with $r\le N$). Then $\|w(uxu^{-1},y) - w(x,y)\| \le 2N\|u-1\| \le 4N \lambda$. Indeed, note that since $x_{i_k},y_{j_k}\in M_1$, we have
    \begin{align*}
        \|w(uxu^{-1},y) - w(x,y)\| &= \|ux_{i_1}u^{-1}y_{j_1} \cdots ux_{i_{r}}u^{-1}y_{j_r} - x_{i_1}y_{j_1}\cdots x_{i_r}y_{j_r}\| \\
        &\le \|ux_{i_1}u^{-1}y_{j_1} \cdots ux_{i_{r}}u^{-1}y_{j_r} - x_{i_1}u^{-1}u_{j_1} \cdots ux_{i_{r}}u^{-1}y_{j_r}\|  \\
        +\ldots &+ \|x_{i_1}y_{j_1} \cdots x_{i_{r}}u^{-1}y_{j_r} - x_{i_1}y_{j_1} \cdots x_{i_{r}}y_{j_r}\| \\
        &\le 2N\|u-1\| \max(\|x_{i_k}\|,\|y_{j_k}\|,\|u\|)^{2r}\\
        &\le 2N\|u-1\|.
    \end{align*}
    A very similar argument, using that $\|u - (1+ih)\| \le \lambda^2$, gives 
    \begin{align*}
        \|w(uxu^{-1},y) - w((1+ih)x(1-ih),y)\| &\le 2N \|u-(1+ih)\|\|(1+ih)\|^{2N} \\
        &\le N2^{2N+1} \lambda^2.
    \end{align*}

    For $\kappa\in \cF_N$, we have that $\kappa(x,y)$ is a sum of at most $2N$ words $w(x,y)$ where $w$ is a monomial of length at most $2N$. Therefore $$\|\kappa(uxu^{-1},y) - \kappa(x,y)\| \le 2N \max_w \|w(uxu^{-1},y) - w(x,y)\| \le 8N^2\lambda.$$ In particular, we have that $\|\kappa(uxu^{-1},y)\|_2 \ge \|\kappa(x,y)\|_2 - 8N^2\lambda$. This implies that $\delta(uxu^{-1},y) \ge \delta(x,y) - 8N^2\lambda$, so we may take $C_1 = 8N^2$.

    Similarly, for $\kappa_1,\kappa_2 \in \cF_N$ distinct, $\kappa_2(x,y)^*\kappa_1(x,y)$ is a sum of $4N^2$ words $w(x,y)$ of length at most $4N$. Therefore $$\|\kappa_2(uxu^{-1},y)^*\kappa_1(uxu^{-1},y) - \kappa_2(x,y)^*\kappa_1(x,y)\| \le 32N^3\lambda.$$ Similarly to $\delta,$ we deduce that $\gamma(uxu^{-1},y) \le \gamma(x,y) + 32N^3\lambda$; we may take $C_3 = 32N^3$.

    Lastly, let us consider what happens for $\ee.$ We observe that for any word $w$ of length at most $2N,$ $w((1+ih)x(1-ih),y)$ splits into a sum of (at most) $2^{2N}$ terms. Most of these terms have at least two occurrences of $h$, and the norm of the sum of all these terms is at most $2^{2N}\lambda^2$. In other words, 
    \begin{align*}
        \|&w((1+ih)x(1-ih),y) - x_{i_1}y_{j_1}\cdots x_{i_r}y_{j_r}  \\
        &- ihx_{i_1}y_{j_1}\cdots x_{i_r}y_{j_r} + ix_{i_1}hy_{j_1}\cdots x_{i_r}y_{j_r} - \ldots + ix_{i_1}y_{j_1}\cdots x_{i_r}hy_{j_r}\| \\ 
        &\le 2^{2N}\lambda^2.
    \end{align*}
    (Note the placement of the $h$'s.) Now we use our construction of $h$. If $\kappa = -\kappa^*$, then 
    \begin{align*}
        \tau(ih\kappa(x,y)) &= -2N\tau(h\xi_\kappa) = -2N\alpha_\kappa = -\tau(x_{i_1}\cdots y_{j_r}).
    \end{align*}
    If $\kappa \neq -\kappa^*$ and $\kappa$ was defined before $-\kappa^\ast$, then 
    \begin{align*}
        \tau(ih\kappa(x,y)) &= 2Ni\tau(h(\xi_\kappa - i\xi_{-\kappa^*})) \\
        &= 2Ni\alpha_\kappa + 2N\alpha_{-\kappa^*}\\
        &= -i\mathrm{Im}(\tau(x_{i_1}y_{j_1}\cdots x_{i_r}y_{j_r})) - \mathrm{Re}(\tau(x_{i_1}y_{j_1}\cdots x_{i_r}y_{j_r})) \\
        &= -\tau(x_{i_1}y_{j_1}\cdots x_{i_r}y_{j_r}).
    \end{align*}
    And lastly if $\kappa \ne -\kappa^\ast$ and $-\kappa^\ast$ was defined before $\kappa$, then a similar computation shows that $\tau(ih \kappa(x,y)) = - \tau(x_{i_1}y_{j_1}\cdots x_{i_r}y_{j_r})$ as well.
    Thus, in all cases, we have 
    \begin{align*}
        \tau(&x_{i_1}y_{j_1}\cdots x_{i_r}y_{j_r} \\&- ihx_{i_1}y_{j_1}\cdots x_{i_r}y_{j_r} + ix_{i_1}hy_{j_1}\cdots x_{i_r}y_{j_r} - \ldots + ix_{i_1}y_{j_1}\cdots x_{i_r}hy_{j_r})\\ &= \tau(x_{i_1}y_{j_1}\cdots x_{i_r}y_{j_r} + ih\kappa(x,y)) = 0.
    \end{align*}
    Therefore $|\tau(w((1+ih)x(1-ih),y))| \le 2^{2N}\lambda^2$, which in turn implies that $$|\tau(w(uxu^{-1},y))|\le (2^{2N} + N2^{2N+1})\lambda^2.$$ Thus we may take $C_2 = 2^{2N} + N2^{2N+1}$ and deduce that $\ee(uxu^{-1},y) \le C_2\lambda^2$.
\end{proof}

\begin{thm}\label{thm-main-technical}
    There exists a constant $C_0$, depending only on $N$, with the following property. Let $(M,\tau)$ be a tracial von Neumann algebra, $x = (x_1,\ldots,x_m)\in M_{\mathrm{sa},1}^m$, and $y = (y_1,\ldots,y_n)\in M_{\mathrm{sa},1}^n$, for some $m,n\ge1.$ Set $\delta = \delta(x,y)$, $\ee= \ee(x,y),$ and $\gamma = \gamma(x,y)$. If $C_0|\cF_N|\sqrt{\ee} < \delta^2 - (2|\cF_N|-1)\gamma$ and $\delta(\delta^2-(2|\cF_N|-1)\gamma)) \ge 8C_1N|\cF_N|\ee$, then  there exists a unitary $u$ in $M$ such that $\ee(uxu^{-1},y)=0$ and $\|u-1\| \le \frac{16N|\cF_N|\ee}{\delta^2 - (2|\cF_N|-1)\gamma} \le \frac{16N}{C_0}\sqrt{\ee}$.
\end{thm}

\begin{proof}
    Let $C_1,C_2,C_3\ge1$ be as in Lemma \ref{lem-main-technical}. Set $C_0 = 8N\sqrt{C_1+C_2+C_3}$. If $\ee=0$ then take $u=1$. Henceforth assume $\ee > 0$. Since $\ee \le 1$ we note the following two inequalities:

    \begin{equation}\label{eqn-def-lambda}
        \frac{4N|\cF_N|\ee}{\delta^2 - (2|\cF_N|-1)\gamma} < \frac{4N|\cF_N|\ee}{C_0|\cF_N|\sqrt{\ee}} \le \frac{4N}{C_0} \le \frac{1}{2}
    \end{equation}
    \begin{equation}\label{eqn-def-C0}
        \frac{8N|\cF_N|C_2}{\frac{C_0^2|\cF_N|}{4N} - 8NC_1 - 2(2|\cF_N|-1)C_3} \le 1.
    \end{equation}
    We also note from the hypothesis $\delta(\delta^2-(2|\cF_N|-1)\gamma)) \ge 8C_1N|\cF_N|\ee$ that
    \begin{equation}\label{eqn-pos-delta}
        \delta - 2C_1\frac{4N|\cF_N|\ee}{\delta^2 - (2|\cF_N|-1)\gamma} \ge 0.
    \end{equation}

    Let $\lambda_0 = 1$ and $u_0 = 1\in M$. We now construct inductively sequences $(u_k)$ of unitaries in $M$ and positive real numbers $\lambda_k$ with the following properties. Set $v_k = u_ku_{k-1}\cdots u_1$ and define $\delta_k = \delta(v_kxv_k^{-1},y),$ $\ee_k = \ee(v_kxv_k^{-1},y)$, and $\gamma_k = \gamma(v_kxv_k^{-1},y)$. Then the following will hold for our choices of $u_k$ and $\lambda_k$:
    \begin{enumerate}
        \item $\delta_k \ge \delta_{k-1} - C_1\lambda_k$
        \item $\ee_k \le C_2\lambda_k^2$
        \item $\gamma_k \le \gamma_{k-1} + C_3\lambda_k$
        \item $\lambda_k\le \frac{1}{2}\lambda_{k-1}$
    \end{enumerate}

    For $k=1,$ setting $\lambda_1 =  \frac{4N|\cF_N|\ee}{\delta^2 - (2|\cF_N|-1)\gamma}$, which by (\ref{eqn-def-lambda}) satisfies $\lambda_1\le \frac12$ and thus satisfies (4). From Lemma \ref{lem-main-technical} we obtain a unitary $u_1\in M$ satisfying (1)--(3) as well. Now assume $u_1,\ldots, u_l$ and $\lambda_1,\ldots,\lambda_l$ have been constructed. Define $\lambda_{l+1} = \frac{4N|\cF_N|\ee_l}{\delta_l^2 - (2|\cF_N|-1)\gamma_l}$. We first claim that $\lambda_{l+1} \le \frac{1}{2}\lambda_l$. 

    For each $1\le k \le l,$ by (\ref{eqn-pos-delta}), (1), and (4) we have that 
    \begin{align*}
        \delta_{k-1} - C_1\lambda_k &\ge \delta_0 - C_1(\lambda_1 + \ldots + \lambda_k)\\
        &\ge \delta - 2C_1\lambda_1\\
        &\ge 0.
    \end{align*}

    Now, for each $1\le k\le l,$ by (1), since $\delta_k \le 2N$, and since $\delta_{k-1}-C_1\lambda_k\ge 0$, we have
    \begin{equation*}
        \delta_k^2 \ge (\delta_{k-1}- C_1\lambda_k)^2 \ge \delta_{k-1}^2 - 2C_1\delta_{k-1}\lambda_k \ge \delta^2_{k-1} - 4NC_1\lambda_k.
    \end{equation*}

    By the previous inequality and (3), we have
    \begin{align*}
        \delta_k^2 - (2|\cF_N|-1)\gamma_k &\ge \delta^2_{k-1} - 4NC_1\lambda_k - (2|\cF_N|-1)\gamma_k \\
        &\ge \delta^2_{k-1} - 4NC_1\lambda_k - (2|\cF_N|-1)(\gamma_{k-1} + C_3\lambda_k) \\
        &= \delta^2_{k-1}  - (2|\cF_N|-1)\gamma_{k-1} - [4NC_1 + (2|\cF_N|-1)C_3]\lambda_k
    \end{align*}

    By applying the previous inequality for all $1\le k\le l$ and item (4), we have (noting that $\delta_0=\delta$ and $\gamma_0=\gamma$)
    \begin{align*}
         \delta_l^2 - (2|\cF_N|-1)\gamma_l &\ge \delta_0^2 - (2|\cF_N|-1)\gamma_0 - [4NC_1 + (2|\cF_N|-1)C_3]\sum_{k=1}^l\lambda_k\\
         &\ge \delta^2 - (2|\cF_N|-1)\gamma - [8NC_1 + 2(2|\cF_N|-1)C_3]\lambda_1 
    \end{align*}

    Now observe that by assumption, $C_0^2|\cF_N|^2\ee < (\delta^2 - (2|\cF_N|-1)\gamma)^2$ so that 
    \begin{align*}
        \delta^2 - (2|\cF_N|-1)\gamma > \frac{C_0^2|\cF_N|^2\ee}{\delta^2 - (2|\cF_N|-1)\gamma} = \frac{C_0^2|\cF_N|}{4N}\lambda_1.
    \end{align*}

    Therefore 
    \begin{align*}
        \delta_l^2 - (2|\cF_N|-1)\gamma_l &\ge \left[\frac{C_0^2|\cF_N|}{4N} - 8NC_1 - 2(2|\cF_N|-1)C_3\right]\lambda_1 \\
        &\ge 8N|\cF_N|C_2 \lambda_1.
    \end{align*}
        Note that this last inequality follows from (\ref{eqn-def-C0}).
    Now using the previous inequality, item (2), and the fact that $\lambda_l\le \lambda_1,$ we compute that
    \begin{align*}
        \lambda_{l+1} = \frac{4N|\cF_N|\ee_l}{\delta^2_l - (2|\cF_N|-1)\gamma_l} \le \frac{4N|\cF_N|C_2\lambda_l^2}{8N|\cF_N|C_2 \lambda_1} \le \frac{\lambda_l}{2}.
    \end{align*}
    This completes the claim. 

    We may now apply Lemma \ref{lem-main-technical} to $v_lxv_l^{-1}$ and $y$ to obtain a new unitary $u_{l+1}$; items (1)--(3) are straightforward to verify. We also note that $\|u_{l+1} - 1\| \le 2\lambda_{l+1}$ by Lemma \ref{lem-main-technical}.

    We now see that the sequence $(v_k)$ is Cauchy in operator norm; indeed, $\|v_{k+1} - v_k\| \le \|u_{k+1}-1\| \le 2\lambda_{k+1} \le 2^{-k}$ by item (4). Let $v$ be the limit of the sequence $(v_k)$, so that $v$ is itself a unitary in $M$. Arguing as in Lemma \ref{lem-main-technical}, we have that $\lim_k \ee(v_kxv_k^{-1},y)= \ee(vxv^{-1},y)$. But the former are bounded by $C_2\lambda_k^2 \to 0$, so $\ee(vxv^{-1},y) = 0.$ We also note that 
    \begin{align*}
        \|v-1\| &\le \sum_{k=0}^\infty \|v_{k+1}-v_k\| \\
        &\le \sum_{k=0}^\infty 2\lambda_{k+1}\\
        &\le 4\lambda_1 \\
        &= \frac{16N|\cF_N|\ee}{\delta^2 - (2|\cF_N|-1)\gamma}, 
    \end{align*}
    proving the last assertion in the statement of the theorem.
\end{proof}

\begin{proof}[Proof of Theorem \ref{mthm-perturb}]
    Let $\widetilde X = \{\operatorname{Re} x, \operatorname{Im}x: x \in X\}$ and $\widetilde Y = \{\operatorname{Re} y, \operatorname{Im}y: y \in Y\}$. Let $C$ and $D$ be the linear spans of $\widetilde X \cup \{xx' - \tau(xx') : x,x'\in \widetilde X\}$ and $\widetilde Y \cup \{yy' - \tau(yy'): y,y'\in \widetilde Y\}$, respectively. Let $x=(x_1,\ldots, x_m)$ and $y=(y_1,\ldots, y_n)$ be orthogonal bases of $E = C+C^*$ and $F = D+D^*$, respectively, such that each $x_i,y_j$ is in $M_{\mathrm{sa},1}$. By Popa's asymptotic freeness theorem \cite{popa1995free} 
    (applied to a separable $\mathrm{II}_1$-subfactor containing $E$ and $F$), there is a unitary $u\in M^\cU$ such that $W^*(uEu^*)$ and $W^*(F)$ are freely independent. First, note that any two words $w(uxu^*,y),w'(uxu^*,y)$ (of the form $x_{i_1}y_{j_1}\cdots x_{i_r}y_{j_r}$) are either equal to each other or orthogonal by free independence. This implies $\delta(uxu^*,y)>0$ and $\gamma(uxu^*,y) = 0$. Furthermore, each of these words $w(uxu^*,y)$ has trace 0, so $\ee(uxu^*,y)=0.$

    Therefore there is a unitary $u'$ in $M$ such that, setting $\delta = \delta(u'xu'^*,y)$, $\ee=\ee(u'xu'^*,y),$ and $\gamma = \gamma(u'xu'^*,y)$, we have $$C_0|\cF_N|\sqrt{\ee} < \delta^2 - (2|\cF_N|-1)\gamma \quad \text{and} \quad \delta(\delta^2-(2|\cF_N|-1)\gamma)) \ge 8C_1N|\cF_N|\ee.$$ By Theorem \ref{thm-main-technical}, we get a unitary $v \in M$ such that $\ee(vu'xu'^*v^*,y) = 0$. By Lemma \ref{lem:trace-trick}(1), we have that $vu'\widetilde Xu'^*v^*$ and $\widetilde Y$ are $N$-independent. It follows that $vu'Xu'^*v^*$ and $ Y$ are $N$-independent as well.
\end{proof}

\begin{proof}[Proof of Theorem \ref{mthm-lift-algs}]
    The proof is nearly the exact same as in \cite[Theorem~5.1]{houdayer2023asymptotic}, 
    replacing 2-independent with $2N$-independent and orthogonal with $N$-independent. 
    As in the proof of \cite[Theorem~5.1]{houdayer2023asymptotic}, we deduce that there are orthogonal bases $x_i = (x_{1,i},\ldots,x_{m_i,i})$ of $C_i$ and $y_i = (y_{1,i},\ldots,y_{n_i,i})$ of $D_i$ of self-adjoint elements such that  $\ee(x_i,y_i) = 0$ for each $i$, 
    which by Lemma \ref{lem:trace-trick}(2) implies each pair of finite-dimensional subalgebras $(C_i,D_i)$ are $N$-independent.
\end{proof}

\begin{proof}[Proof of Theorem \ref{mthm-sri-problem}]
    
    First suppose that $\cU$ is a countably cofinal ultrafilter. Let $u,v\in M^\cU$ be freely independent Haar unitaries. Since $W^*(u)$ and $W^*(v)$ are abelian and separable, Theorem \ref{mthm-lift-algs} provides for each $N$ sequences of unitaries $(u_{n,N})_n$ and $(v_{n,N})_n$ such that: 
    \begin{enumerate}
        \item $(u_{n,N})_n$ is a lift of $u$;
        \item $(v_{n,N})_n$ is a lift of $v$;
        \item $(u_{n,N})_n$ and $(v_{n,N})_n$ are $N$-independent for all $n \in A_N$ (for some set $A_N\in\cU$ possibly dependent on $N$). 
    \end{enumerate}
    Since for each $N$ $(u_{n,N})_n$ is a lift of $u$ and $(v_{n,N})_n$ is a lift of $v$, for all $N$ there is a set $B_N\in\cU$ such that for all $n\in B_N$, $\|u_{n,N} - u_{n,1}\| < \frac1N$ and $\|v_{n,N}-v_{n,1}\|<\frac1N$. 

    Since $\cU$ is countably cofinal, there is a sequence of sets $C_N \in\cU$ such that for each $N$, $C_N \subset A_N\cap B_N$ and the $C_N$ are strictly decreasing and have empty intersection. Define
    \begin{align*}
        u_n &= u_{n,N} \text{ for } n\in C_N\setminus C_{N+1} \\
        v_n &= v_{n,N} \text{ for } n\in C_N\setminus C_{N+1}.
    \end{align*}
    It remains now to verify that $(u_n)_n$ is a lift of $u$, $(v_n)_n$ is a lift of $v$, and that for any $N$, $u_n$ and $v_n$ are $N$-independent for all $n\in C_N$. For the former two statements, note that for $ n\in C_N$, there is some $M\ge N$ such that $n \in C_M\setminus C_{M+1}$, so we have $\|u_n-u_{n,1}\| < \frac1M \le \frac1N$. Therefore, as $n\to\cU$ we have $\|u_n - u_{n,1}\| \to 0$ which means that $(u_n)_n = (u_{n,1})_n = u$, and similarly for $(v_n)_n$. For the latter, note that for $n\in C_N$ we again have that for some $M\ge N,$ $n\in C_{M}\setminus C_{M+1}$ so that $W^*(u_n)$ and $W^*(v_n)$ are $M$-independent and thus $N$-independent.

    On the other hand, if $\cU$ is countably complete, then by \cite[Lemma~2.3]{boutonnet:hal-01450072} (see also \cite[Proposition~6.1(2)]{ge2001ultraproducts}) the diagonal embedding $M\to M^\cU$ is an isomorphism. Then two freely independent Haar unitaries $u,v\in M^\cU = M$ are already freely independent, hence $N$-independent for all $N$.
\end{proof}

\bibliographystyle{amsalpha}
\bibliography{inneramen}

\end{document}